\documentclass[12pt]{article}
\usepackage[leqno]{amsmath}
\usepackage{amsthm}
\usepackage{amssymb,amscd,latexsym,stmaryrd}
\usepackage{paralist}
\usepackage[colorlinks,pdfpagelabels,pdfstartview = FitH,bookmarksopen = true,bookmarksnumbered = true,linkcolor = black,plainpages = false,hypertexnames = false,citecolor = black]{hyperref}
\usepackage{bm}
\usepackage{fonttable}                                                         
       
\usepackage{tocloft}
\usepackage{aurical}
\usepackage[T1]{fontenc}
\usepackage[bbgreekl]{mathbbol}
\usepackage[all]{xy}

\newcommand{\C}{{\mathbb{C}}}
\newcommand{\F}{{\mathbb{F}}}

\newcommand{\Ha}{\mathbb{H}}

\newcommand{\Q}{{\mathbb{Q}}}
\newcommand{\oQ}{\overline{\Q}}

\newcommand{\R}{{\mathbb{R}}}
\newcommand{\Z}{{\mathbb{Z}}}

\newcommand{\id}{\mathrm{id}}
\newcommand{\ind}{\mathrm{ind}}

\newcommand{\ord}{\mathrm{ord}}

\newcommand{\spec}{\mathrm{spec}\,}

\newcommand{\End}{\mathrm{End}\,}
\newcommand{\Ext}{\mathrm{Ext}}

\newcommand{\Hom}{\mathrm{Hom}}

\newcommand{\Kt}{\mathrm{Kt}}

\newcommand{\oalpha}{\overline{\alpha}}

\newcommand{\Res}{\mathrm{Res}}
\newcommand{\RRe}{\mathrm{Re}\,}

\newcommand{\tr}{\mathrm{tr}}

\newcommand{\Ch}{{\mathcal C}}

\newcommand{\Rh}{{\mathcal R}}

\newcommand{\eo}{\mathfrak{o}}

\newcommand{\oV}{\overline{V}}

\newcommand{\oF}{\overline{\mathbb{F}}}

\newcommand{\silo}{\xrightarrow{\sim}}

\newcommand{\hullet}{{\scriptscriptstyle \bullet\,}}
\newcommand{\verk}{\mbox{\scriptsize $\,\circ\,$}}

\newcommand{\halb}{\frac{1}{2}}

\newtheorem{theorem}{Theorem}[section]
\newtheorem{punkt}[theorem]{$\!\!$}
\begin{document}
\title{There is no ``Weil-''cohomology theory with {\it real} coefficients for arithmetic curves}
\author{Christopher Deninger\footnote{Funded by the Deutsche Forschungsgemeinschaft (DFG, German Research Foundation) under Germany's Excellence Strategy EXC 2044--390685587, Mathematics M\"unster: Dynamics--Geometry--Structure and the CRC 1442 Geometry: Deformations and Rigidity} }
\date{\ }
\maketitle
\section{Weil cohomologies over $\oF_p$} \label{sec:1}
A well known argument by Serre shows that there is no Weil cohomology theory with real coefficients for smooth projective varieties over $\oF_p$. For such a cohomology theory and every elliptic curve $E$ over $\oF_p$ there would have to be an $\R$-algebra anti-homomorphism
\begin{punkt} \quad 
\label{1.1}
$(\End E) \otimes \R \longrightarrow \End H^1 (E , \R) \; ,$
\end{punkt}
However, if $E$ is supersingular, then $(\End E) \otimes \Q$ is a quaternion algebra over $\Q$ which is non-split precisely at $p$ and $\infty$. Thus $(\End E) \otimes \R$ would be the Hamilton quaternions $\Ha$ whereas $\End H^1 (E , \R) = M_2 (\R)$, and there is no anti-homomorphism of $\R$-algebras $\Ha \to M_2 (\R)$. The algebras $(\End E) \otimes \Q_l$ for $l \neq p$ and $(\End E) \otimes Q_p$ are split, where $Q_p$ is the quotient field of $W (\oF_p)$, the ($p$-typical) Witt vector ring of $\oF_p$. Hence Serre's argument does not contradict the existence of Weil cohomology theories with values in $\Q_l$ resp. $Q_p$. Indeed $l$-adic and crystalline cohomology are such theories. In particular, embedding $\Q_l$ or $Q_p$ into $\C$ one sees that there is a Weil cohomology with values in $\C$-vector spaces. Of course such embeddings are highly uncanonical. However, if the standard conjectures on algebraic cycles hold, there is a Weil-cohomology theory with values in $\oQ$. It would give a theory with complex coefficients after choosing an embedding $\oQ \hookrightarrow \C$, which is a more reasonable process than embedding $\Q_l$ or $Q_p$ into $\C$. 

In Scholze \cite[Conjecture 9.5 ff]{Sch} the existence of a Weil cohomology theory $H^n_{Kt_{\Q}} (X)$ for varieties $X$ over $\oF_p$ with values in Kottwitz' category $\Kt_{\Q}$ is conjectured, and assuming well known conjectures also proved. As yet, we have no indication that such a refined category as (the $\ind$-category of $\Kt_{\Q}$) may be the target of our conjectural ``Weil-''cohomology $H^n (\overline{\spec \eo_k} , \Ch)$ for arithmetic curves in the next section. However, the following much simpler $\R$-linear abelian $\otimes$-category $K_{\R}$ may be a target category. Its objects are $\Z / 2$-graded $\C$-vector spaces $V = V_0 \oplus V_1$ with an antilinear isomorphism $\sigma$ respecting the grading and $\sigma^2 = (-1)^i$ on $V_i$. Morphisms are $\C$-linear maps which respect the grading and commute with $\sigma$. On $V_0$ the involution $\sigma$ defines the real structure $V^{\sigma = \id}_0$, i.e. $V_0 = V^{\sigma = \id}_0 \otimes_{\R} \C$. On $V_1$ we obtain a quaternionic structure: The quaternions $\Ha = \R \oplus \R i \oplus \R j \oplus \R k$ act on $V_1$ by $1 , i , j = \sigma$ and $k = i \sigma$, turning $V_1$ into an $\Ha$-vector space. Thus Serre's argument would not exclude a Weil cohomology theory $H^n$ with values in $K_{\R}$ if $H^1 (E) = H^1 (E)_1$ for supersingular elliptic curves. In fact, as Milne pointed out to me, it follows directly from Grothendieck's standard conjectures that there exists a Weil cohomology with values in $K_{\R}$ - the conjectures imply that the motives over $\oF_p$  form a polarized Tannakian category, to which one can apply \cite[Theorem 4.29]{DM}. He also remarked that this was likely known to Grothendieck in the 1960s. Note that there are canonical $\otimes$-functors $\Kt_{\Q} \to \Kt_{\R} \to K_{\R}$, where Kottwitz' category $\Kt_{\R}$ is the same as the category $V_{\infty}$ in \cite{M}, Example 1.7.

\section{``Weil-''cohomologies over $\spec \Z$} \label{sec:2}
In several articles we discussed the possibility of a cohomology theory $H^n (\ , \Ch)$ with values in $\C$-vector spaces which is defined for all arithmetic schemes over $\spec \Z$ and their arithmetic compactifications. The theory should be equipped with functorial endomorphisms $\theta$ whose eigenvalues would be the zeroes and poles of the corresponding zeta functions. In particular the cohomology groups can be infinite dimensional. For arithmetic schemes over $\spec \F_p$ such a theory can be constructed by a formal process starting with $l$-adic cohomology, c.f. \cite[section 4]{D7}. The question arises whether a theory of this type can have real coefficients. For arithmetic schemes which are not flat over spec $\Z$ the arguments in \cite[4.7]{D1} make this improbable but for flat schemes over $\spec \Z$ real coefficients remained a theoretical possibility. In the following, we will explain why no real-valued theory of the type conjectured in \cite{D2}, \cite{D3} can exist, even for spectra of number rings. On the other hand, a theory with values in $K_{\R}$ refining the conjectural $\C$-theory seems plausible. We begin by recalling some of the discussion from \cite{D2}, \cite{D3}. Let $k / \Q$ be a number field and let $\overline{\spec \eo_k}$ be the arithmetic compactification of $\spec \eo_k$. The basic expected properties of $H^n (\overline{\spec \eo_k} , \Ch)$ and its operator $\theta$ are the following: 

\begin{punkt} \label{2.1} \rm $\theta$ is a derivation with respect to cup-product. \end{punkt}
\begin{punkt} \label{2.2} \rm $H^0 (\overline{\spec \eo_k} , \Ch) = \C$ with $\theta = 0$. \end{punkt}
\begin{punkt} \label{2.3} \rm $H^1 (\overline{\spec \eo_k} , \Ch)$ is infinite dimensional with the eigenvalues of $\theta$ being the non-trivial zeroes of $\zeta_k (s)$ with their multiplicities. Here the eigenvalues $\alpha$ are counted with their algebraic multiplicities $\dim H^1 (\overline{\spec \eo_k} , \Ch)^{\theta \sim \alpha}$. Here 
\[
H^{\theta \sim \alpha} = \{h \in H \mid (\theta - \alpha)^n (h) = 0 \quad \text{for some} \; n \ge 1 \} 
\]
is the generalized $\alpha$-eigenspace of $\theta$. 
\end{punkt}
\begin{punkt} \rm \label{2.4} There is a $\theta$-equivariant trace isomorphism
\[
\tr : H^2 (\overline{\spec \eo_k} , \Ch) \silo \C (-1) \quad \text{where} \; \C (-1) = \C \; \text{with} \; \theta = \id \; .
\]
\end{punkt}
By \ref{2.1} and \ref{2.4}, for $h_1 , h_2 \in H^1 (\overline{\spec \eo_k} , \Ch)$ we would have
\begin{punkt} \label{2.5} \quad
$h_1 \cup h_2 = \theta (h_1 \cup h_2) = \theta h_1 \cup h_2 + h_1 \cup \theta h_2 \; .$
\end{punkt}
For $h_1 \in H^1 (\overline{\spec \eo_k} , \Ch)^{\theta \sim \alpha}$ and $h_2 \in H^1 (\overline{\spec \eo_k} , \Ch)^{\theta \sim \beta}$ we have
\[
h_1 \cup h_2 \in H^2 (\overline{\spec \eo_k} , \Ch)^{\theta \sim \alpha + \beta} \; .
\]
Hence $h_1 \cup h_2 \neq 0$ can only happen by \ref{2.4} if $\alpha + \beta = 1$ and we expect that:
\begin{punkt} \rm \label{2.6} the pairing
\[
\cup : H^1 (\overline{\spec \eo_k} , \Ch)^{\theta \sim \alpha} \times H^1 (\overline{\spec \eo_k} , \Ch)^{\theta \sim 1 - \alpha} \xrightarrow{\cup} H^2 (\overline{\spec \eo_k} , \Ch) \overset{\tr}{\silo} \C
\]
is perfect. In particular, $\alpha$ is an eigenvalue of $\theta$ on $H^1$ if and only if $1 - \alpha$ is an eigenvalue and the (algebraic) multiplicities are equal. Note that via \ref{2.3} this is compatible with the functional equation of $\zeta_k (s)$. 
\end{punkt} 

Much more optimistically, one could hope for the following property.

\begin{punkt} \rm \label{2.7}
There is a canononical anti-linear automorphism $\ast$ of $H^1 (\overline{\spec \eo_k} , \Ch)$ with $\ast^2 = - \id , \ast \verk \theta = \theta \verk \ast$ and such that we get a scalar product by setting
\[
\langle h , h' \rangle = \tr (h \cup \ast h') \quad \text{on} \; H^1 (\overline{\spec \eo_k} ,\Ch) \; .
\]
\end{punkt}
In particular $\alpha$ is then an eigenvalue of $\theta$ on $H^1$ if and only if $\oalpha$ is an eigenvalue. Note that \ref{2.5} and \ref{2.7} imply the relation 
\[
\langle h, h' \rangle = \langle \theta h , h' \rangle + \langle h , \theta h' \rangle \; .
\]
Hence $\theta - \halb$ would be skew-symmetric on $H^1 (\overline{\spec \eo_k} , \Ch)$ and its eigenvalues therefore purely imaginary. Using \ref{2.3} the Riemann hypotheses would follow. Of course this argument follows the lead of Serre \cite{S}. Also note that if \ref{2.7} is true, skew-symmetry of $\theta - \halb$ would imply that the algebraic eigenspaces $H^{\theta \sim \alpha}$ of $\theta$ are the usual ones $H^{\theta = \alpha}$. 

The following heuristic argument suggests that property \ref{2.7} may be more than whishful thinking. In \cite{D8} we showed that higher dimensional analogues of \ref{2.7} together with certain Abel-Jacobi isomorphisms would imply Beilinson's conjectures on the positivity of height pairings. In particular, suitably normalized the N\'eron-Tate height pairing on $E (\Q)$ for an elliptic curve $E$ over $\Q$ would be the scalar product as in \ref{2.7} restricted to the $1$-eigenspace of $\theta$ on the relevant cohomology group. We expect an Abel-Jacobi isomorphism to identify $E (\Q) \otimes \C$ with this $1$-eigenspace. Thus the same scalar product on cohomology whose positivity would imply that all zeroes of $L (E,s)$ lie on the line $\RRe s = 1$ would also imply that the N\'eron-Tate height pairing on $E (\Q)$ is positive  definite, c.f. \cite{D8} for the details. 

For $\overline{\spec \eo_k}$ there is no natural known candidate of a $\C$-vector space which might be isomorphic to $H^1 (\overline{\spec \eo_k})^{\theta \sim \halb}$ via an Abel-Jacobi map. We will discuss this intriguing problem below.

As a particular case of \ref{2.3} we would get the formula
\begin{punkt} \label{2.8} \quad 
$\ord_{s = 1/2} \, \zeta_k (s) = \dim_{\C} H^1 (\overline{\spec \eo_k} , \Ch)^{\theta \sim \halb} \; .$
\end{punkt}
Note that if an antilinear involution $\ast$ on $H^1 (\overline{\spec \eo_k} , \Ch)$ exists with $\ast^2 = - \id$ and $\ast \verk \theta = \theta \verk \ast$ then $V_k = H^1 (\overline{\spec \eo_k} , \Ch)^{\theta \sim \halb}$ carries a quaternionic structure and hence it has even $\C$-dimension. Since the order of $\zeta_k (s)$ at $s = 1/2$ is even by the functional equation, this is compatible with \eqref{2.8}. Alternatively, the existence of the perfect alternating pairing \ref{2.6} for $\alpha = 1/2$
\begin{punkt}
\label{2.9} \quad 
$\cup : V_k \times V_k \longrightarrow \C$
\end{punkt}
would also imply that $\dim V_k$ is even.

Let us now assume that $k / \Q$ is Galois with group $G$. By contravariant functoriality of the cohomology ring with derivation $(H^{\hullet} (\overline{\spec \eo_k} , \Ch) ,\cup ,  \theta)$ the group $G$ would act on $H^1 (\overline{\spec \eo_k} , \Ch)$ and in particular on $V_k$. By the following argument kindly pointed out to me by Baptiste Morin, it is natural to assume that $G$ acts trivially on $H^2 (\overline{\spec \eo_k} , \Ch) \cong \C$, so that the pairings \ref{2.6} and in particular \eqref{2.9} would be $G$-invariant. Namely one may expect a natural continuous $G$-equivariant homomorphism from the Arakelov Chow group:
\begin{equation} \label{eq:1}
cl : CH^1 (\overline{\spec \eo_k}) \longrightarrow H^2 (\overline{\spec \eo_k} , \Ch) \; .
\end{equation}
The map $cl$ would send the compact kernel of the Arakelov degree map to zero since $\C$ has no non-trivial compact subgroups. Hence $cl$ would factor over the Arakelov degree map and this would lead to a copy of $\R$ with trivial $G$-action inside of $H^2 (\overline{\spec \eo_k} , \Ch)$. Hence $G$ would act trivially on the entire group.

For an equivalence class $\pi$ of irreducible representations of $G$ over $\C$ let $e_{\pi} \in \C [G]^{\mathrm{opp}} = \End_G \C [G]$ be the corresponding idempotent projecting to the $\pi$-isotypical component of $\C [G]$. For any complex vector space $H$ on which $G$ acts, $e_{\pi}$ defines a projector and we set $H (\pi) = e_{\pi} H$. For a realization $V_{\pi}$ of $\pi$ consider the $G$-equivariant isomorphism:
\[
\Hom_G (V_{\pi} , H) \otimes V_{\pi} \silo H (\pi) \quad \text{with} \; f \otimes v \mapsto f (v) \; .
\]
Since $\theta$ commutes with the $G$-action, it respects both $H^1 (\overline{\spec \eo_k} , \Ch) (\pi)$ and the decomposition
\[
H^1 (\overline{\spec \eo_k} , \Ch) = \bigoplus_{\pi} H^1 (\overline{\spec \eo_k} , \Ch) (\pi) \; .
\]
Consider the $G$-isomorphism:
\[
\Hom_G (V_{\pi} , H^1 (\overline{\spec \eo_k} , \Ch)) \otimes V_{\pi} \silo H^1 (\overline{\spec \eo_k} , \Ch) (\pi) \; .
\]
It is compatible with $\theta$ if $\theta$ acts on the left as $\theta (f \otimes v) = (\theta \verk f) \otimes v$. In accordance with the formula
\[
\zeta_k (s) = \prod_{\pi} L (\C [G] (\pi) , s) = \prod_{\pi} L (V_{\pi} , s)^{\dim V_{\pi}}
\]
we expect:
\begin{punkt}
\rm \label{2.10}
The eigenvalues $\alpha$ of $\theta$ on $\Hom_G (V_{\pi} , H^1 (\overline{\spec \eo_k} , \Ch))$ counted with their algebraic multiplicities are the non-trivial zeroes of the Artin $L$-function $L (V_{\pi} , s)$ counted with multiplicity. 
\end{punkt}
A corresponding assertion in the function field case is true, c.f. \cite[\S\,9]{FQ}.

In particular, we would get
\begin{punkt} \quad
\label{2.11}
$\dim_{\C} \Hom_G (V_{\pi} , H^1 (\overline{\spec \eo_k} , \Ch)^{\theta \sim \alpha}) = \ord_{s = \alpha} L (V_{\pi} , s) \; .$
\end{punkt}
For $\alpha = 1/2$ this amounts to the formula
\begin{punkt} \quad
\label{2.12}
$\dim_{\C} \Hom_G (V_{\pi} , V_k) = \ord_{s = 1/2} L (V_{\pi} , s) \; .$
\end{punkt}

\begin{theorem}
\label{t2.13}
For a number field $k$ there is no functorial real valued cohomology theory $H^{\hullet} (\overline{\spec \eo_k} , \Rh)$ with endomorphism $\theta_{\R}$ such that $H^{\hullet} (\overline{\spec \eo_k} , \Ch) := H^{\hullet} (\overline{\spec \eo_k} , \Rh) \otimes_{\R} \C$ with $\theta := \theta_{\R} \otimes \id$ satisfies property \ref{2.12}.
\end{theorem}

\begin{proof}
Let $k$ be a finite Galois extension of $\Q$ whose Galois group $G$ has an irreducible symplectic representation $V_{\pi}$ with root number $W (\pi) = -1$. For example, there exist quaternion fields i.e. normal extensions of $\Q$ with Galois group the quaternion group $H_8$ with $8$ elements whose Artin $L$-function for the unique non-abelian representation (it is symplectic) has root number $-1$. See \cite{O} for examples of such fields. From the functional equation of $L (V_{\pi} , s)$ it follows that $\ord_{s = 1/2} L (V_{\pi} , s) \ge 1$ is odd. By property \ref{2.12} the representation $V_{\pi}$ occurs with odd multiplicity $\Hom_G (V_{\pi} , V_k)$ in $V_k$. If a theory $H^{\hullet} (\overline{\spec \eo_k} , \Rh)$ as in the theorem exists set 
\[
V^{\R}_k = H^1 (\overline{\spec \eo_k} , \Rh)^{\theta \sim \halb} \; .
\]
Then $V^{\R}_k$ would be a $G$-module with $V_k = V^{\R}_k \otimes_{\R} \C$. and we could write $V^{\R}_k$ as a sum of irreducible representations $W$ of $G$ over $\R$. According to \cite[13.2]{S1} there are three possibilities for $W$. If $W$ has the form $W = \Res V$ for an irreducible non-self dual irreducible representation $V$ over $\C$, then we would have $W \otimes \C = V \oplus \oV$ and since both $V$ and $\oV$ are not self-dual they cannot be symplectic. If $W$ is orthogonal then $W \otimes \C$ is an irreducible orthogonal $\C$-representation and hence not symplectic. If $W$ is symplectic i.e. $W = \Res V$ for a symplectic irreducible representation $V$ over $\C$ then $W \otimes \C = V \oplus V$. It follows that the symplectic representation $V_{\pi}$ occurs with even multiplicity in $V_k$. This is a contradiction.
\end{proof}

I would like to thank one of the referees for suggesting this simplification of my original argument. 

There is thus no real valued ``Weil-''cohomology for arithmetic curves. On the other hand, the operator $\ast$ in \ref{2.7} would give a canonical quaternionic structure on $H^1 (\overline{\spec \eo_k} , \Ch)$ and hence we could view $(H^1 (\overline{\spec \eo_k} , \Ch) , \ast)$ as an object of $K_{\R}$, equal to its degree $1$ part. The groups $H^0$ and $H^2$ being isomorphic to $\C$ by \ref{2.2} and \ref{2.4} must be equal to their degree zero parts if they actually take values in $K_{\R}$. In other words, they must have natural real structures. For $H^2$ the real structure would be the image of the cycle class map \eqref{eq:1} in Morin's heuristic argument above. We noted that $cl$ would have to factorize over the Arakelov degree map onto $\R$. 

Hence, a ``Weil-''cohomology theory with values in $K_{\R}$ is a possibility for arithmetic curves. 

In \cite{D3}, \cite{D4} we discussed analogies between the leafwise (or tangential) cohomology of certain foliated dynamical systems and cohomologies such as the hoped for $H^1 (\overline{\spec \eo_k} , \Ch)$. The operator $\theta$ corresponds to the infinitesimal generator of the induced $\R$-action by the flow on leafwise cohomology. Complex leafwise cohomology always has the real structure given by the cohomology of the real forms. Since such a real structure cannot exist on the cohomologies $H^1 (\overline{\spec \eo_k} , \Ch)$ as we have seen, it follows that the analogy is too simple. There has to be a twisted version of leafwise cohomology with complex coefficients that does not have a natural real structure. This can become relevant since in \cite{D5} we did construct foliated dynamical systems for arithmetic schemes which have some (but not all) the expected properties.

We will now state the expected properties of the spaces $V_k = H^1 (\overline{\spec \eo_k} , \Ch)^{\theta \sim 1/2}$ and suggest the problem to find a direct number theoretical construction of the $V_k$. This refines the discussion in \cite[\S\,7]{D9}. We need the 
following conjecture which is due to Serre (unpublished). Sarnak and Rubinstein also arrived at this conjecture in a wider context on p. 195 of \cite{RS}. Previously Chowla had conjectured that the Dirichlet $L$-series of quadratic characters do not vanish at $s = 1/2$. 

{\bf Vanishing Conjecture} {\it For an irreducible complex representation of the Galois group of $\Q$ the corresponding Artin $L$-function has a zero at $s = 1/2$ if and only if the representation is symplectic with root number $-1$. In this case the order of vanishing is one.} 

If the root number is $-1$, the functional equation implies that the Artin $L$-function vanishes at $s = 1/2$. For orthogonal irreducible representations it is known that the root number is $+1$ by \cite{FQ}. Hence there is no a priori reason why their $L$-functions should vanish at $s = 1/2$. For certain quaternion fields the vanishing conjecture has been verified numerically in \cite{O}.

If $k / \Q$ is Galois, it follows from \ref{2.12} and the vanishing conjecture that precisely the symplectic $\pi$'s with $W (\pi) = -1$ among the irreducible representations of $G$ are realized in the $G$-module $V_k$ and each of them with multiplicity one. Thus, as $\C [G]$-modules we should have
\begin{punkt} \quad
\label{2.13}
$V_k = \bigoplus_{\pi \, \text{\rm sympl} \atop W (\pi) = -1} V_k (\pi)$ \quad {\rm and} \quad $V_k (\pi) \cong V_{\pi}$.
\end{punkt}
Since the $V_k (\pi)$ are irreducible self dual and pairwise non-isomorphic, it would follow that the cup-product on $V_k$ gives the (up to scalar) unique $G$-invariant symplectic pairing on $V_k (\pi)$. 

All together we get the following suggestion:

\begin{punkt} \quad \label{2.14}
There should be a covariant functor $k \mapsto (V_k , \cup)$ from the category of number fields $k$ to the category of finite dimensional complex vector spaces $V_k$ with a non-degenerate alternating form $\cup : V_k \times V_k \to \C$. If $k$ is Galois over $\Q$ with group $G$, the $\C [G]$-module structure on $V_k$ resulting from functoriality should be the one described in \ref{2.13}. In particular, we should have
\[
\dim_{\C} V_k = \sum_{\pi \, \text{sympl} \atop W (\pi) = -1} \dim V_{\pi} \; .
\]
Moreover, we expect the symplectic vector spaces $(V_k , \cup)$ to be equipped with an antilinear isomorphism $\ast : V_k \to V_k$ with $\ast^2 = -1$ such that we obtain a scalar product by setting 
\[
\langle v,w \rangle = v \cup \ast w \quad \text{for} \; v,w \in V_k \; .
\]
Like height pairings, the symplectic form $\cup$ should be a sum of local contributions.\end{punkt}

In case a natural candidate for $V_k$ can be found, one may also expect a formula for the leading coefficient of $\zeta_k (s)$ at $s = 1/2$ in the style of the Birch Swinnerton-Dyer or Lichtenbaum Weil-\'etale conjectures. From the point of view of the Beilinson conjectures, $V_k$ would be the $1/2$-eigenspace of the Adams operators on $K_0 (k)$, or the $CH^{1/2}$ group of $\spec k$ or the group of extensions $\Ext^1 (\C (0) , \C (1/2))$ in some $\C$-linear category of ``motives'' over $k$ where a $1/2$-twist of the Tate motive exists. Classically all these objects do not make sense in the number field case. In the function field case the $H^1$ of a supersingular elliptic curve over a finite field can be used as a kind of $1/2$-Tate twist, c.f.\cite{R}. In the Galois module theory of rings of integers in tamely ramified extensions $k / \Q$, symplectic representations with root number $-1$ play a vital role. Also, exponential motives come to mind because they allow non-integral twists.

{\bf Acknowledgements.} I would like to thank the referees for their suggestions to simplify and improve this note. I am also grateful to Luca Barbieri Viale, Werner Bley, James Milne, Baptiste Morin, Peter Sarnak and Jean-Pierre Serre for valuable mathematical and historical comments. Many thanks go to Umberto Zannier for the invitation to the SNS in Pisa, where the first version of this note was written.


\begin{thebibliography}{10}

\bibitem{DM}
P.~Deligne and J.~S. Milne.
\newblock {\em Tannakian Categories}, pages 101--228.
\newblock Springer Berlin Heidelberg, Berlin, Heidelberg, 1982.

\bibitem{D7}
Christopher Deninger.
\newblock Motivic {$L$}-functions and regularized determinants.
\newblock In {\em Motives ({S}eattle, {WA}, 1991)}, volume~55 of {\em Proc.
  Sympos. Pure Math.}, pages 707--743. Amer. Math. Soc., Providence, RI, 1994.

\bibitem{D9}
Christopher Deninger.
\newblock Motivic {$L$}-functions and regularized determinants. {II}.
\newblock In {\em Arithmetic geometry ({C}ortona, 1994)}, Sympos. Math.,
  XXXVII, pages 138--156. Cambridge Univ. Press, Cambridge, 1997.

\bibitem{D2}
Christopher Deninger.
\newblock Some analogies between number theory and dynamical systems on
  foliated spaces.
\newblock In {\em Proceedings of the {I}nternational {C}ongress of
  {M}athematicians, {V}ol. {I} ({B}erlin, 1998)}, number Extra Vol. I, pages
  163--186, 1998.

\bibitem{D1}
Christopher Deninger.
\newblock On dynamical systems and their possible significance for arithmetic
  geometry.
\newblock In {\em Regulators in analysis, geometry and number theory}, volume
  171 of {\em Progr. Math.}, pages 29--87. Birkh\"{a}user Boston, Boston, MA,
  2000.

\bibitem{D3}
Christopher Deninger.
\newblock Number theory and dynamical systems on foliated spaces.
\newblock {\em Jahresber. Deutsch. Math.-Verein.}, 103(3):79--100, 2001.

\bibitem{D4}
Christopher Deninger.
\newblock Analogies between analysis on foliated spaces and arithmetic
  geometry.
\newblock In {\em Groups and analysis}, volume 354 of {\em London Math. Soc.
  Lecture Note Ser.}, pages 174--190. Cambridge Univ. Press, Cambridge, 2008.

\bibitem{D8}
Christopher Deninger.
\newblock The {H}ilbert-{P}olya strategy and height pairings.
\newblock In {\em Casimir force, {C}asimir operators and the {R}iemann
  hypothesis}, pages 275--283. Walter de Gruyter, Berlin, 2010.

\bibitem{D5}
Christopher Deninger.
\newblock Dynamical systems for arithmetic schemes, 2022.
\newblock arXiv:1807.06400.

\bibitem{FQ}
A.~Fr\"{o}hlich and J.~Queyrut.
\newblock On the functional equation of the {A}rtin {$L$}-function for
  characters of real representations.
\newblock {\em Invent. Math.}, 20:125--138, 1973.

\bibitem{M}
J.~S. Milne.
\newblock Motives over finite fields.
\newblock In {\em Motives ({S}eattle, {WA}, 1991)}, volume~55 of {\em Proc.
  Sympos. Pure Math.}, pages 401--459. Amer. Math. Soc., Providence, RI, 1994.

\bibitem{O}
Sami Omar.
\newblock On {A}rtin {$L$}-functions for octic quaternion fields.
\newblock {\em Experiment. Math.}, 10(2):237--245, 2001.

\bibitem{R}
Niranjan Ramachandran.
\newblock Values of zeta functions at {$s=1/2$}.
\newblock {\em Int. Math. Res. Not.}, (25):1519--1541, 2005.

\bibitem{RS}
Michael Rubinstein and Peter Sarnak.
\newblock Chebyshev's bias.
\newblock {\em Experiment. Math.}, 3(3):173--197, 1994.

\bibitem{Sch}
Peter Scholze.
\newblock {$p$}-adic geometry.
\newblock In {\em Proceedings of the {I}nternational {C}ongress of
  {M}athematicians---{R}io de {J}aneiro 2018. {V}ol. {I}. {P}lenary lectures},
  pages 899--933. World Sci. Publ., Hackensack, NJ, 2018.

\bibitem{S}
Jean-Pierre Serre.
\newblock Analogues k\"{a}hl\'{e}riens de certaines conjectures de {W}eil.
\newblock {\em Ann. of Math. (2)}, 71:392--394, 1960.

\bibitem{S1}
Jean-Pierre Serre.
\newblock {\em Linear representations of finite groups}.
\newblock Graduate Texts in Mathematics, Vol. 42. Springer-Verlag, New
  York-Heidelberg, 1977.
\newblock Translated from the second French edition by Leonard L. Scott.

\end{thebibliography}

\end{document}